\documentclass[11pt]{article}

\usepackage[a4paper,margin=0.85in]{geometry}
\usepackage[T1]{fontenc}
\usepackage{lmodern}
\usepackage{amsmath,amssymb,amsthm}

\newtheorem{theorem}{Theorem}
\usepackage{booktabs,longtable,array}
\usepackage{enumitem}
\usepackage{hyperref}
\usepackage{microtype}

\title{\textbf{Sequence Factory:}\\
Families of Integer Sequences Generated by G\"odel Implication}
\author{Volkan Yildiz\\\texttt{vo1kan@hotmail.co.uk}}
\date{}
\hypersetup{pdfkeywords={Godel implication, finite chains, Catalan enumeration, integer sequences, generating functions, root-pair factorisation, OEIS}}

\begin{document}

\maketitle

\vspace{ 3cm}

\begin{abstract}
A finite G\"odel chain generates not one integer sequence but a large
family of them. The combinatorial mechanism is developed in a companion
paper; here the sequences themselves are the primary objects.

For
\[
V_m=\{v_0<\cdots<v_{m-1}\},
\]
let \(g^{(m)}_{j,n}\) count outputs equal to \(v_j\), aggregated over
all full bracketings on \(n\) variables and all \(m^n\) valuations.
The root-pair refinement produces
\[
N^{(m)}_{i,j}(n)
=
\sum_{r=1}^{n-1}
g^{(m)}_{i,r}g^{(m)}_{j,n-r},
\qquad
N^{(m)}_{i,j}(x)=G^{(m)}_i(x)G^{(m)}_j(x).
\]
The initial coefficients of these products identify the unordered root-value pair uniquely; hence there are \emph{exactly} \(\binom{m+1}{2}\) distinct pair-count generating functions for every fixed \(m\).

We catalogue the factories for \(m=4,5,\ldots,10\), giving the first ten
terms of every base sequence and every symmetric root-pair sequence.
This range contains 49 base sequences and 210 symmetric
pair classes, for 259 indexed sequence families before further
aggregates are included. The note is intended as a reproducible source
for sequence analysis and OEIS comparison.
\end{abstract}

\vspace{10em}

\noindent\rule{\textwidth}{0.4pt}


\vspace{0.5em}

\clearpage
\section{Introduction}

The companion paper establishes the Catalan recurrences and the root-pair
factorisation. This note changes the emphasis: the structural theory is
the machine, while the integer sequences produced by that machine are the
objects of study.

For \(m=4,\ldots,10\), the immediate catalogue sizes are
\[
\begin{array}{c|c|c|c}
m & \text{base} & \text{symmetric pair classes} & \text{indexed total}\\
\hline
4 & 4 & 10 & 14 \\
5 & 5 & 15 & 20 \\
6 & 6 & 21 & 27 \\
7 & 7 & 28 & 35 \\
8 & 8 & 36 & 44 \\
9 & 9 & 45 & 54 \\
10 & 10 & 55 & 65 \\
\hline
\text{Total} & 49 & 210 & 259
\end{array}
\]

Thus even before upper tails, diagonals, fixed-gap sums, or higher
products are introduced, the first seven factories already produce more
than two hundred pair-convolution sequences. In addition to recording
these coefficients, Theorem~\ref{thm:distinct-pairs} proves that the
\(\binom{m+1}{2}\) pair products are distinct for each fixed chain;
swap symmetry is their only source of coincidence.

\section{The factory rule}

For \(0\le j\le m-1\), let
\[
g^{(m)}_{j,n}
\]
be the total number of occurrences of truth value \(v_j\) among all output
columns of all fully bracketed G\"odel implications on \(n\) variables.
Set
\[
G^{(m)}_j(x)=\sum_{n\ge1}g^{(m)}_{j,n}x^n.
\]

The initial conditions are
\[
g^{(m)}_{j,1}=1.
\]
For \(n\ge2\) and \(0\le j\le m-2\),
\begin{equation}
g^{(m)}_{j,n}
=
\sum_{r=1}^{n-1}
\left(
\sum_{p=j+1}^{m-1}g^{(m)}_{p,r}
\right)
g^{(m)}_{j,n-r}.
\end{equation}
The top level is
\begin{equation}
g^{(m)}_{m-1,n}
=
m^nC_{n-1}
-
\sum_{j=0}^{m-2}g^{(m)}_{j,n}.
\end{equation}

For \(0\le i,j\le m-1\), define
\begin{equation}
N^{(m)}_{i,j}(1)=0,
\qquad
N^{(m)}_{i,j}(n)
=
\sum_{r=1}^{n-1}
g^{(m)}_{i,r}g^{(m)}_{j,n-r}.
\end{equation}
Then
\[
N^{(m)}_{i,j}(x)
=
G^{(m)}_i(x)G^{(m)}_j(x)
\]
and \(N^{(m)}_{i,j}=N^{(m)}_{j,i}\).

For \(n\ge2\), the level sequences are recovered by
\[
g^{(m)}_{j,n}
=
\sum_{i=j+1}^{m-1}N^{(m)}_{j,i}(n)
\qquad (j<m-1),
\]
and
\[
g^{(m)}_{m-1,n}
=
\sum_{0\le i\le j\le m-1}N^{(m)}_{i,j}(n).
\]
At \(n=1\), each level instead contributes the initial term
\(g^{(m)}_{j,1}=1\), whereas all root-pair counts are zero.
Equivalently, for every \(m\ge2\) the valid generating-function
identities are
\[
G^{(m)}_j(x)-x=\sum_{i=j+1}^{m-1}N^{(m)}_{j,i}(x)
\qquad (j<m-1),
\]
and
\[
G^{(m)}_{m-1}(x)-x=
\sum_{0\le i\le j\le m-1}N^{(m)}_{i,j}(x).
\]

\section{Distinctness of the root-pair sequences}

The symmetry \(N^{(m)}_{i,j}(x)=N^{(m)}_{j,i}(x)\) is immediate from
commutativity of the product of generating functions. The following
theorem proves that no other coincidences arise within a fixed chain.

\begin{theorem}[Exact number of distinct pair products]\label{thm:distinct-pairs}
For every fixed integer \(m\ge2\) and indices
\(0\le i,j,k,\ell\le m-1\),
\[
G^{(m)}_i(x)G^{(m)}_j(x)=G^{(m)}_k(x)G^{(m)}_\ell(x)
\quad\Longleftrightarrow\quad
\{i,j\}=\{k,\ell\}.
\]
Consequently, the factory on \(V_m\) produces exactly
\(\binom{m+1}{2}\) distinct root-pair generating functions.
\end{theorem}

\begin{proof}
Throughout the proof, \(m\) is fixed and its superscript is omitted.
Because \(g_{j,1}=1\), every component has an expansion
\[
G_j(x)=x+a_jx^2+b_jx^3+O(x^4).
\]
For \(0\le j<m-1\), the recurrence at \(n=2\) gives
\[
a_j=m-1-j,
\]
so the non-top values of \(a_j\) are precisely \(1,2,\ldots,m-1\),
all distinct. As the total number of two-variable rows is \(m^2\),
the top coefficient is
\[
a_{m-1}=m^2-\sum_{j=0}^{m-2}a_j
       =\frac{m(m+1)}2=:M.
\]
For \(j<m-1\) put \(a=a_j\). The \(n=3\) recurrence has
contributions from splits of sizes \(1+2\) and \(2+1\), giving
\[
\begin{aligned}
b_j
&=\underbrace{a\cdot a}_{1+2}
  +\underbrace{\sum_{p=j+1}^{m-1}g_{p,2}}_{2+1}\\
&=a^2+M+\sum_{r=1}^{a-1}r
 =\frac{3a^2-a+m(m+1)}2.
\end{aligned}
\]

Now consider a product of two \emph{non-top} components
\(N_{i,j}(x)=G_i(x)G_j(x)\). Define
\(S=a_i+a_j\) and \(P=a_ia_j\). Multiplication gives
\[
[x^3]N_{i,j}(x)=S
\]
and
\[
\begin{aligned}
[x^4]N_{i,j}(x)
 &=b_i+b_j+a_ia_j\\
 &=\frac{3S^2-S+2m(m+1)-4P}{2}.
\end{aligned}
\]
Hence these two coefficients determine both \(S\) and \(P\).
The numbers \(a_i,a_j\) are therefore determined, up to order,
as the roots of \(z^2-Sz+P=0\). Since the \(a_j\) are distinct
for \(j<m-1\), the unordered pair \(\{i,j\}\) is uniquely determined.

It remains to compare pairs containing the top index \(m-1\).
Their \(x^3\)-coefficients are \(M+a_i\) when exactly one index is
top, and \(2M\) for the top--top pair. Since
\[
M+1>2(m-1),\qquad M>m-1\qquad(m\ge2),
\]
all these coefficients are strictly larger than those of any
non-top pair. They are also mutually distinct: the values
\(a_0,\ldots,a_{m-2}\) are distinct, and \(2M>M+a_i\).
Thus the only equality between pair products is the obvious
swap symmetry. Counting unordered pairs with repetition proves
the final assertion.
\end{proof}

\section{Computation and verification}

All coefficients in this catalogue were generated recursively, not by
constructing complete truth tables for large \(n\). For every
\(m=4,\ldots,10\) and \(1\le n\le10\), we checked exactly that
\[
\sum_{j=0}^{m-1}g^{(m)}_{j,n}=m^nC_{n-1}.
\]
For \(2\le n\le10\) we also checked the root-pair reconstruction
identities
\[
g^{(m)}_{j,n}=\sum_{i=j+1}^{m-1}N^{(m)}_{j,i}(n)
\quad (j<m-1),
\]
and
\[
g^{(m)}_{m-1,n}
=\sum_{0\le i\le j\le m-1}N^{(m)}_{i,j}(n).
\]
The initial values \(g^{(m)}_{j,1}=1\) and
\(N^{(m)}_{i,j}(1)=0\) were verified separately.

\section{The G\"odel--4 factory}
For \(m=4\), the factory has 4 base sequences and
\(\binom{5}{2}=10\) symmetric root-pair classes.

\begin{table}[ht]
\centering
\scriptsize
\begin{tabular}{@{}l p{0.80\textwidth}@{}}
\toprule
Sequence & Initial terms, \(n=1,\ldots,10\)\\
\midrule
$g^{(4)}_{0,n}$ & 1, 3, 22, 211, 2306, 27230, 338444, 4362627, 57788170, 781825066 \\
$g^{(4)}_{1,n}$ & 1, 2, 15, 143, 1559, 18379, 228175, 2938786, 38902987, 526061137 \\
$g^{(4)}_{2,n}$ & 1, 1, 11, 101, 1116, 13186, 164093, 2116676, 28052479, 379673121 \\
$g^{(4)}_{3,n}$ & 1, 10, 80, 825, 9355, 113237, 1431976, 18696855, 250122284, 3410617188 \\
$g^{(4)}_n$ & 4, 16, 128, 1280, 14336, 172032, 2162688, 28114944, 374865920, 5098176512 \\
\bottomrule
\end{tabular}
\caption{G\"odel--4 base sequences.}
\end{table}

The complete pair catalogue is given in Appendix~\ref{app:m4}.

\section{The G\"odel--5 factory}
For \(m=5\), the factory has 5 base sequences and
\(\binom{6}{2}=15\) symmetric root-pair classes.

\begin{table}[ht]
\centering
\scriptsize
\begin{tabular}{@{}l p{0.80\textwidth}@{}}
\toprule
Sequence & Initial terms, \(n=1,\ldots,10\)\\
\midrule
$g^{(5)}_{0,n}$ & 1, 4, 37, 445, 6089, 89963, 1398597, 22545676, 373433705, 6316990813 \\
$g^{(5)}_{1,n}$ & 1, 3, 27, 321, 4366, 64270, 996699, 16038492, 265297395, 4483081441 \\
$g^{(5)}_{2,n}$ & 1, 2, 20, 238, 3249, 47914, 743935, 11981021, 198302337, 3352542347 \\
$g^{(5)}_{3,n}$ & 1, 1, 16, 181, 2511, 37101, 577493, 9315136, 154363464, 2612110206 \\
$g^{(5)}_{4,n}$ & 1, 15, 150, 1940, 27535, 417002, 6595776, 107697800, 1801571849, 30715743943 \\
$g^{(5)}_n$ & 5, 25, 250, 3125, 43750, 656250, 10312500, 167578125, 2792968750, 47480468750 \\
\bottomrule
\end{tabular}
\caption{G\"odel--5 base sequences.}
\end{table}

The complete pair catalogue is given in Appendix~\ref{app:m5}.

\section{The G\"odel--6 factory}
For \(m=6\), the factory has 6 base sequences and
\(\binom{7}{2}=21\) symmetric root-pair classes.

\begin{table}[ht]
\centering
\scriptsize
\begin{tabular}{@{}l p{0.80\textwidth}@{}}
\toprule
Sequence & Initial terms, \(n=1,\ldots,10\)\\
\midrule
$g^{(6)}_{0,n}$ & 1, 5, 56, 811, 13340, 236766, 4420176, 85548675, 1701026756, 34539693466 \\
$g^{(6)}_{1,n}$ & 1, 4, 43, 613, 10001, 176621, 3286305, 63450868, 1259360255, 25535543671 \\
$g^{(6)}_{2,n}$ & 1, 3, 33, 471, 7690, 135862, 2528553, 48828798, 969264735, 19655246407 \\
$g^{(6)}_{3,n}$ & 1, 2, 26, 370, 6075, 107602, 2006041, 38785013, 770575551, 15636782477 \\
$g^{(6)}_{4,n}$ & 1, 1, 22, 295, 4929, 87417, 1633619, 31630264, 629126334, 12777371598 \\
$g^{(6)}_{5,n}$ & 1, 21, 252, 3920, 66829, 1215284, 23076858, 452311646, 9081751649, 185841910093 \\
$g^{(6)}_n$ & 6, 36, 432, 6480, 108864, 1959552, 36951552, 720555264, 14411105280, 293986547712 \\
\bottomrule
\end{tabular}
\caption{G\"odel--6 base sequences.}
\end{table}

The complete pair catalogue is given in Appendix~\ref{app:m6}.

\section{The G\"odel--7 factory}
For \(m=7\), the factory has 7 base sequences and
\(\binom{8}{2}=28\) symmetric root-pair classes.

\begin{table}[ht]
\centering
\scriptsize
\begin{tabular}{@{}l p{0.80\textwidth}@{}}
\toprule
Sequence & Initial terms, \(n=1,\ldots,10\)\\
\midrule
$g^{(7)}_{0,n}$ & 1, 6, 79, 1339, 25739, 533519, 11628371, 262697790, 6096273451, 144459442837 \\
$g^{(7)}_{1,n}$ & 1, 5, 63, 1049, 19976, 411682, 8937997, 201354400, 4662858223, 110310814213 \\
$g^{(7)}_{2,n}$ & 1, 4, 50, 830, 15783, 325000, 7052249, 158812295, 3676659097, 86961163899 \\
$g^{(7)}_{3,n}$ & 1, 3, 40, 667, 12723, 262443, 5701205, 128488406, 2976355674, 70428785986 \\
$g^{(7)}_{4,n}$ & 1, 2, 33, 545, 10464, 216389, 4709105, 106260638, 2463720998, 58339618919 \\
$g^{(7)}_{5,n}$ & 1, 1, 29, 449, 8779, 181679, 3962533, 89530435, 2077919862, 49242222006 \\
$g^{(7)}_{6,n}$ & 1, 28, 392, 7126, 141834, 3010546, 66716216, 1525955665, 35751870705, 853652612778 \\
$g^{(7)}_n$ & 7, 49, 686, 12005, 235298, 4941258, 108707676, 2473099629, 57705658010, 1373394660638 \\
\bottomrule
\end{tabular}
\caption{G\"odel--7 base sequences.}
\end{table}

The complete pair catalogue is given in Appendix~\ref{app:m7}.

\section{The G\"odel--8 factory}
For \(m=8\), the factory has 8 base sequences and
\(\binom{9}{2}=36\) symmetric root-pair classes.

\begin{table}[ht]
\centering
\scriptsize
\begin{tabular}{@{}l p{0.80\textwidth}@{}}
\toprule
Sequence & Initial terms, \(n=1,\ldots,10\)\\
\midrule
$g^{(8)}_{0,n}$ & 1, 7, 106, 2059, 45302, 1074182, 26774292, 691591027, 18348557438, 497043869170 \\
$g^{(8)}_{1,n}$ & 1, 6, 87, 1659, 36139, 851623, 21137859, 544343022, 14408800755, 389619444541 \\
$g^{(8)}_{2,n}$ & 1, 5, 71, 1345, 29208, 687058, 17033365, 438285224, 11594422527, 313371548957 \\
$g^{(8)}_{3,n}$ & 1, 4, 58, 1102, 23967, 564264, 13997009, 360300199, 9534215745, 257747365875 \\
$g^{(8)}_{4,n}$ & 1, 3, 48, 915, 19979, 471371, 11708925, 301694310, 7989114018, 216096161122 \\
$g^{(8)}_{5,n}$ & 1, 2, 41, 769, 16912, 399917, 9950969, 256701118, 6803625094, 184154594831 \\
$g^{(8)}_{6,n}$ & 1, 1, 37, 649, 14539, 343903, 8574901, 221460787, 5874990574, 159130751326 \\
$g^{(8)}_{7,n}$ & 1, 36, 576, 11982, 272706, 6617730, 167646744, 4383049977, 117377624889, 3203369012466 \\
$g^{(8)}_n$ & 8, 64, 1024, 20480, 458752, 11010048, 276824064, 7197425664, 191931351040, 5220532748288 \\
\bottomrule
\end{tabular}
\caption{G\"odel--8 base sequences.}
\end{table}

The complete pair catalogue is given in Appendix~\ref{app:m8}.

\section{The G\"odel--9 factory}
For \(m=9\), the factory has 9 base sequences and
\(\binom{10}{2}=45\) symmetric root-pair classes.

\begin{table}[ht]
\centering
\scriptsize
\begin{tabular}{@{}l p{0.80\textwidth}@{}}
\toprule
Sequence & Initial terms, \(n=1,\ldots,10\)\\
\midrule
$g^{(9)}_{0,n}$ & 1, 8, 137, 3001, 74381, 1985835, 55717089, 1619779872, 48361587245, 1474196113141 \\
$g^{(9)}_{1,n}$ & 1, 7, 115, 2473, 60674, 1609574, 44963535, 1303023676, 38810916587, 1180840164097 \\
$g^{(9)}_{2,n}$ & 1, 6, 96, 2046, 49981, 1322566, 36885123, 1067676561, 31773725169, 966090353251 \\
$g^{(9)}_{3,n}$ & 1, 5, 80, 1705, 41655, 1102345, 30745477, 890006204, 26487440556, 805386985958 \\
$g^{(9)}_{4,n}$ & 1, 4, 67, 1435, 35154, 931743, 26013269, 753548944, 22437977634, 682529813733 \\
$g^{(9)}_{5,n}$ & 1, 3, 57, 1221, 30041, 797945, 22309611, 646908489, 19276865922, 586707438052 \\
$g^{(9)}_{6,n}$ & 1, 2, 50, 1048, 25984, 691544, 19365842, 562165762, 16765405297, 510591807008 \\
$g^{(9)}_{7,n}$ & 1, 1, 46, 901, 22756, 605596, 16991710, 493764751, 14737764127, 449126095831 \\
$g^{(9)}_{8,n}$ & 1, 45, 810, 18975, 486060, 13273374, 378360252, 11130169050, 335359616733, 10297276986591 \\
$g^{(9)}_n$ & 9, 81, 1458, 32805, 826686, 22320522, 631351908, 18467043309, 554011299270, 16952745757662 \\
\bottomrule
\end{tabular}
\caption{G\"odel--9 base sequences.}
\end{table}

The complete pair catalogue is given in Appendix~\ref{app:m9}.

\section{The G\"odel--10 factory}
For \(m=10\), the factory has 10 base sequences and
\(\binom{11}{2}=55\) symmetric root-pair classes.

\begin{table}[ht]
\centering
\scriptsize
\begin{tabular}{@{}l p{0.80\textwidth}@{}}
\toprule
Sequence & Initial terms, \(n=1,\ldots,10\)\\
\midrule
$g^{(10)}_{0,n}$ & 1, 9, 172, 4195, 115664, 3433718, 107101112, 3460865331, 114845229880, 3890670682738 \\
$g^{(10)}_{1,n}$ & 1, 8, 147, 3521, 96101, 2834665, 88025749, 2835307552, 93856339495, 3173500495891 \\
$g^{(10)}_{2,n}$ & 1, 7, 125, 2963, 80454, 2365894, 73321505, 2358337346, 77985080287, 2634710562867 \\
$g^{(10)}_{3,n}$ & 1, 6, 106, 2506, 67971, 1997646, 61886873, 1990065881, 65795607939, 2222597639521 \\
$g^{(10)}_{4,n}$ & 1, 5, 90, 2135, 58005, 1706405, 52897847, 1701771884, 56282734406, 1901737600838 \\
$g^{(10)}_{5,n}$ & 1, 4, 77, 1835, 50014, 1473953, 45744919, 1472850124, 48740812794, 1647662981883 \\
$g^{(10)}_{6,n}$ & 1, 3, 67, 1591, 43561, 1286425, 39980471, 1288506169, 42671115032, 1443280777782 \\
$g^{(10)}_{7,n}$ & 1, 2, 60, 1388, 38314, 1133364, 35276562, 1138062282, 37717589307, 1276484845208 \\
$g^{(10)}_{8,n}$ & 1, 1, 56, 1211, 34046, 1006776, 31393110, 1013739321, 33622911257, 1138570926131 \\
$g^{(10)}_{9,n}$ & 1, 55, 1100, 28655, 815870, 24761154, 784371852, 25640494110, 858482579603, 29290783487141 \\
$g^{(10)}_n$ & 10, 100, 2000, 50000, 1400000, 42000000, 1320000000, 42900000000, 1430000000000, 48620000000000 \\
\bottomrule
\end{tabular}
\caption{G\"odel--10 base sequences.}
\end{table}

The complete pair catalogue is given in Appendix~\ref{app:m10}.

\section{Aggregate families}

Two natural pair aggregates are
\[
S^{(m)}_n
=
\sum_{0\le i<j\le m-1}N^{(m)}_{i,j}(n)
=
g^{(m)}_n-g^{(m)}_{m-1,n},
\]
and
\[
D^{(m)}_n
=
\sum_{i=0}^{m-1}N^{(m)}_{i,i}(n)
=
2g^{(m)}_{m-1,n}-m^nC_{n-1}
\]
for \(n\ge2\).

\begin{longtable}{@{}l p{0.80\textwidth}@{}}
\toprule
Sequence & Initial terms, \(n=1,\ldots,10\)\\
\midrule
\endfirsthead
\toprule
Sequence & Initial terms, \(n=1,\ldots,10\)\\
\midrule
\endhead
$S^{(4)}_n$ & 0, 6, 48, 455, 4981, 58795, 730712, 9418089, 124743636, 1687559324 \\
$D^{(4)}_n$ & 0, 4, 32, 370, 4374, 54442, 701264, 9278766, 125378648, 1723057864 \\
$S^{(5)}_n$ & 0, 10, 100, 1185, 16215, 239248, 3716724, 59880325, 991396901, 16764724807 \\
$D^{(5)}_n$ & 0, 5, 50, 755, 11320, 177754, 2879052, 47817475, 810174948, 13951019136 \\
$S^{(6)}_n$ & 0, 15, 180, 2560, 42035, 744268, 13874694, 268243618, 5329353631, 108144637619 \\
$D^{(6)}_n$ & 0, 6, 72, 1360, 24794, 471016, 9202164, 184068028, 3752398018, 77697272474 \\
$S^{(7)}_n$ & 0, 21, 294, 4879, 93464, 1930712, 41991460, 947143964, 21953787305, 519742047860 \\
$D^{(7)}_n$ & 0, 7, 98, 2247, 48370, 1079834, 24724756, 578811701, 13798083400, 333910564918 \\
$S^{(8)}_n$ & 0, 28, 448, 8498, 186046, 4392318, 109177320, 2814375687, 74553726151, 2017163735822 \\
$D^{(8)}_n$ & 0, 8, 128, 3484, 86660, 2225412, 58469424, 1568674290, 42823898738, 1186205276644 \\
$S^{(9)}_n$ & 0, 36, 648, 13830, 340626, 9047148, 252991656, 7336874259, 218651682537, 6655468771071 \\
$D^{(9)}_n$ & 0, 9, 162, 5145, 145434, 4226226, 125368596, 3793294791, 116707934196, 3641808215520 \\
$S^{(10)}_n$ & 0, 45, 900, 21345, 584130, 17238846, 535628148, 17259505890, 571517420397, 19329216512859 \\
$D^{(10)}_n$ & 0, 10, 200, 7310, 231740, 7522308, 248743704, 8380988220, 286965159206, 9961566974282 \\
\bottomrule
\end{longtable}

\section{OEIS viewpoint}

Each pair class is OEIS-ready in the sense that it has a precise offset,
a reproducible recurrence, and a product generating function. A failed
short-prefix search should not be interpreted as proof of novelty; longer
prefixes, shifts, deleted initial zeros, and standard convolution
transforms should be checked.

The global family is
\[
\mathcal F=
\left\{
N^{(m)}_{i,j}:
m\ge2,\ 0\le i\le j\le m-1
\right\}.
\]
The finite catalogue through \(m=10\) is only the first visible part of
this sequence factory.

\section{Conclusion}

Extending the catalogue from \(m=4\) through \(m=10\) yields 210
indexed symmetric pair classes and 49 base sequences. In each fixed
chain, Theorem~\ref{thm:distinct-pairs} shows that the pair classes
represent distinct generating functions, with swap symmetry as their
only possible coincidence. The combinatorial structure explains the
machine; this catalogue records what it produces.

\appendix

\section{Complete G\"odel--4 pair catalogue}
\label{app:m4}

\begin{longtable}{@{}l p{0.80\textwidth}@{}}
\toprule
Class & Initial terms, \(n=1,\ldots,10\)\\
\midrule
\endfirsthead
\toprule
Class & Initial terms, \(n=1,\ldots,10\)\\
\midrule
\endhead
$N^{(4)}_{0,0}(n)$ & 0, 1, 6, 53, 554, 6362, 77580, 986253, 12927170, 173452334 \\
$N^{(4)}_{0,1}(n)$ & 0, 1, 5, 43, 443, 5046, 61209, 775277, 10134321, 135696192 \\
$N^{(4)}_{0,2}(n)$ & 0, 1, 4, 36, 367, 4178, 50613, 640554, 8368030, 111992206 \\
$N^{(4)}_{0,3}(n)$ & 0, 1, 13, 132, 1496, 18006, 226622, 2946796, 39285819, 534136668 \\
$N^{(4)}_{1,1}(n)$ & 0, 1, 4, 34, 346, 3915, 47284, 597085, 7787516, 104093243 \\
$N^{(4)}_{1,2}(n)$ & 0, 1, 3, 28, 281, 3185, 38444, 485351, 6328829, 84580645 \\
$N^{(4)}_{1,3}(n)$ & 0, 1, 12, 115, 1278, 15194, 189731, 2453435, 32574158, 441480492 \\
$N^{(4)}_{2,2}(n)$ & 0, 1, 2, 23, 224, 2555, 30826, 389311, 5077062, 67857384 \\
$N^{(4)}_{2,3}(n)$ & 0, 1, 11, 101, 1116, 13186, 164093, 2116676, 28052479, 379673121 \\
$N^{(4)}_{3,3}(n)$ & 0, 1, 20, 260, 3250, 41610, 545574, 7306117, 99586900, 1377654903 \\
\bottomrule
\end{longtable}

\section{Complete G\"odel--5 pair catalogue}
\label{app:m5}

\begin{longtable}{@{}l p{0.80\textwidth}@{}}
\toprule
Class & Initial terms, \(n=1,\ldots,10\)\\
\midrule
\endfirsthead
\toprule
Class & Initial terms, \(n=1,\ldots,10\)\\
\midrule
\endhead
$N^{(5)}_{0,0}(n)$ & 0, 1, 8, 90, 1186, 17107, 261568, 4165509, 68356600, 1147871987 \\
$N^{(5)}_{0,1}(n)$ & 0, 1, 7, 76, 985, 14073, 213856, 3391055, 55471185, 929224925 \\
$N^{(5)}_{0,2}(n)$ & 0, 1, 6, 65, 837, 11920, 180757, 2862017, 46766696, 782765098 \\
$N^{(5)}_{0,3}(n)$ & 0, 1, 5, 57, 727, 10361, 157014, 2485333, 40601030, 679430909 \\
$N^{(5)}_{0,4}(n)$ & 0, 1, 19, 247, 3540, 53609, 846970, 13807271, 230594794, 3925569881 \\
$N^{(5)}_{1,1}(n)$ & 0, 1, 6, 63, 804, 11387, 172070, 2717823, 44330730, 740970784 \\
$N^{(5)}_{1,2}(n)$ & 0, 1, 5, 53, 673, 9511, 143509, 2264357, 36905831, 616501792 \\
$N^{(5)}_{1,3}(n)$ & 0, 1, 4, 46, 577, 8173, 123293, 1945519, 31709130, 529689571 \\
$N^{(5)}_{1,4}(n)$ & 0, 1, 18, 222, 3116, 46586, 729897, 11828616, 196682434, 3336890078 \\
$N^{(5)}_{2,2}(n)$ & 0, 1, 4, 44, 556, 7850, 118344, 1866130, 30400866, 507649223 \\
$N^{(5)}_{2,3}(n)$ & 0, 1, 3, 38, 471, 6680, 100714, 1588826, 25889409, 432390625 \\
$N^{(5)}_{2,4}(n)$ & 0, 1, 17, 200, 2778, 41234, 643221, 10392195, 172412928, 2920151722 \\
$N^{(5)}_{3,3}(n)$ & 0, 1, 2, 33, 394, 5640, 85016, 1342301, 21881472, 365572659 \\
$N^{(5)}_{3,4}(n)$ & 0, 1, 16, 181, 2511, 37101, 577493, 9315136, 154363464, 2612110206 \\
$N^{(5)}_{4,4}(n)$ & 0, 1, 30, 525, 8380, 135770, 2242054, 37725712, 645205280, 11188954483 \\
\bottomrule
\end{longtable}

\section{Complete G\"odel--6 pair catalogue}
\label{app:m6}

\begin{longtable}{@{}l p{0.80\textwidth}@{}}
\toprule
Class & Initial terms, \(n=1,\ldots,10\)\\
\midrule
\endfirsthead
\toprule
Class & Initial terms, \(n=1,\ldots,10\)\\
\midrule
\endhead
$N^{(6)}_{0,0}(n)$ & 0, 1, 10, 137, 2182, 37926, 697764, 13359813, 263454382, 5314590026 \\
$N^{(6)}_{0,1}(n)$ & 0, 1, 9, 119, 1863, 32058, 585953, 11167469, 219471717, 4415727228 \\
$N^{(6)}_{0,2}(n)$ & 0, 1, 8, 104, 1615, 27666, 504237, 9591178, 188222046, 3782831750 \\
$N^{(6)}_{0,3}(n)$ & 0, 1, 7, 92, 1423, 24343, 443229, 8424869, 165248498, 3319796736 \\
$N^{(6)}_{0,4}(n)$ & 0, 1, 6, 83, 1272, 21787, 396530, 7536395, 147804133, 2969073638 \\
$N^{(6)}_{0,5}(n)$ & 0, 1, 26, 413, 7167, 130912, 2490227, 48828764, 980280362, 20052264114 \\
$N^{(6)}_{1,1}(n)$ & 0, 1, 8, 102, 1570, 26755, 485968, 9221433, 180642808, 3625507031 \\
$N^{(6)}_{1,2}(n)$ & 0, 1, 7, 88, 1345, 22833, 413728, 7837595, 153347793, 3074848217 \\
$N^{(6)}_{1,3}(n)$ & 0, 1, 6, 77, 1173, 19900, 360373, 6824057, 133476032, 2675747158 \\
$N^{(6)}_{1,4}(n)$ & 0, 1, 5, 69, 1039, 17669, 319926, 6059017, 118518278, 2375987585 \\
$N^{(6)}_{1,5}(n)$ & 0, 1, 25, 379, 6444, 116219, 2192278, 42730199, 854018152, 17408960711 \\
$N^{(6)}_{2,2}(n)$ & 0, 1, 6, 75, 1140, 19295, 348950, 6601659, 129039786, 2585504860 \\
$N^{(6)}_{2,3}(n)$ & 0, 1, 5, 65, 985, 16675, 301525, 5703809, 111478943, 2233460884 \\
$N^{(6)}_{2,4}(n)$ & 0, 1, 4, 58, 865, 14701, 265853, 5031067, 98352306, 1970804959 \\
$N^{(6)}_{2,5}(n)$ & 0, 1, 24, 348, 5840, 104486, 1961175, 38093922, 759433486, 15450980564 \\
$N^{(6)}_{3,3}(n)$ & 0, 1, 4, 56, 844, 14306, 258744, 4895290, 95684994, 1917136391 \\
$N^{(6)}_{3,4}(n)$ & 0, 1, 3, 50, 735, 12536, 226762, 4293050, 83944497, 1682384977 \\
$N^{(6)}_{3,5}(n)$ & 0, 1, 23, 320, 5340, 95066, 1779279, 34491963, 686631054, 13954397500 \\
$N^{(6)}_{4,4}(n)$ & 0, 1, 2, 45, 634, 10932, 197672, 3745973, 73282224, 1469263503 \\
$N^{(6)}_{4,5}(n)$ & 0, 1, 22, 295, 4929, 87417, 1633619, 31630264, 629126334, 12777371598 \\
$N^{(6)}_{5,5}(n)$ & 0, 1, 42, 945, 18424, 361802, 7213066, 146243860, 3010293824, 62785270663 \\
\bottomrule
\end{longtable}

\section{Complete G\"odel--7 pair catalogue}
\label{app:m7}

\begin{longtable}{@{}l p{0.80\textwidth}@{}}
\toprule
Class & Initial terms, \(n=1,\ldots,10\)\\
\midrule
\endfirsthead
\toprule
Class & Initial terms, \(n=1,\ldots,10\)\\
\midrule
\endhead
$N^{(7)}_{0,0}(n)$ & 0, 1, 12, 194, 3626, 73787, 1587468, 35518653, 818161076, 19273463003 \\
$N^{(7)}_{0,1}(n)$ & 0, 1, 11, 172, 3161, 63681, 1360980, 30308327, 695704677, 16344502053 \\
$N^{(7)}_{0,2}(n)$ & 0, 1, 10, 153, 2785, 55808, 1188693, 26409873, 605184820, 14199378106 \\
$N^{(7)}_{0,3}(n)$ & 0, 1, 9, 137, 2483, 49641, 1055770, 23432581, 536556306, 12581928613 \\
$N^{(7)}_{0,4}(n)$ & 0, 1, 8, 124, 2240, 44758, 951412, 21108646, 483209709, 11328555017 \\
$N^{(7)}_{0,5}(n)$ & 0, 1, 7, 114, 2041, 40842, 867913, 19255680, 440768378, 10333117472 \\
$N^{(7)}_{0,6}(n)$ & 0, 1, 34, 639, 13029, 278789, 6203603, 142182683, 3334849561, 79671961576 \\
$N^{(7)}_{1,1}(n)$ & 0, 1, 10, 151, 2728, 54411, 1155298, 25610191, 585870350, 13728197480 \\
$N^{(7)}_{1,2}(n)$ & 0, 1, 9, 133, 2381, 47255, 1000241, 22125773, 505375475, 11828090200 \\
$N^{(7)}_{1,3}(n)$ & 0, 1, 8, 118, 2105, 41701, 881649, 19486735, 444834430, 10406464171 \\
$N^{(7)}_{1,4}(n)$ & 0, 1, 7, 106, 1885, 37342, 789295, 17442556, 398118226, 9312606342 \\
$N^{(7)}_{1,5}(n)$ & 0, 1, 6, 97, 1706, 33876, 715940, 15823989, 361198447, 8449421945 \\
$N^{(7)}_{1,6}(n)$ & 0, 1, 33, 595, 11899, 251508, 5550872, 126475347, 2953331645, 70314231555 \\
$N^{(7)}_{2,2}(n)$ & 0, 1, 8, 116, 2060, 40706, 859264, 18971698, 432742362, 10117644543 \\
$N^{(7)}_{2,3}(n)$ & 0, 1, 7, 102, 1807, 35664, 752234, 16599306, 378471769, 8845965289 \\
$N^{(7)}_{2,4}(n)$ & 0, 1, 6, 91, 1607, 31737, 669451, 14773299, 336845156, 7873107886 \\
$N^{(7)}_{2,5}(n)$ & 0, 1, 5, 83, 1445, 28638, 604098, 13335825, 304127198, 7109432392 \\
$N^{(7)}_{2,6}(n)$ & 0, 1, 32, 554, 10924, 228961, 5026466, 114103865, 2657214974, 63132658332 \\
$N^{(7)}_{3,3}(n)$ & 0, 1, 6, 89, 1574, 31048, 654584, 14439797, 329151964, 7691711875 \\
$N^{(7)}_{3,4}(n)$ & 0, 1, 5, 79, 1391, 27476, 579481, 12786297, 291508471, 6812835733 \\
$N^{(7)}_{3,5}(n)$ & 0, 1, 4, 72, 1243, 24676, 520485, 11490828, 262053872, 6125903529 \\
$N^{(7)}_{3,6}(n)$ & 0, 1, 31, 516, 10089, 210291, 4601239, 104211281, 2422793331, 57490046724 \\
$N^{(7)}_{4,4}(n)$ & 0, 1, 4, 70, 1222, 24197, 510604, 11271415, 257045130, 6008644784 \\
$N^{(7)}_{4,5}(n)$ & 0, 1, 3, 64, 1085, 21643, 456712, 10089253, 230178823, 5382327174 \\
$N^{(7)}_{4,6}(n)$ & 0, 1, 30, 481, 9379, 194746, 4252393, 96171385, 2233542175, 52957291745 \\
$N^{(7)}_{5,5}(n)$ & 0, 1, 2, 59, 956, 19297, 406958, 8999207, 205406860, 4804946091 \\
$N^{(7)}_{5,6}(n)$ & 0, 1, 29, 449, 8779, 181679, 3962533, 89530435, 2077919862, 49242222006 \\
$N^{(7)}_{6,6}(n)$ & 0, 1, 56, 1568, 36204, 836388, 19550580, 464000740, 11169705658, 272285957142 \\
\bottomrule
\end{longtable}

\section{Complete G\"odel--8 pair catalogue}
\label{app:m8}

\begin{longtable}{@{}l p{0.80\textwidth}@{}}
\toprule
Class & Initial terms, \(n=1,\ldots,10\)\\
\midrule
\endfirsthead
\toprule
Class & Initial terms, \(n=1,\ldots,10\)\\
\midrule
\endhead
$N^{(8)}_{0,0}(n)$ & 0, 1, 14, 261, 5602, 130666, 3219100, 82430637, 2172302362, 58531291838 \\
$N^{(8)}_{0,1}(n)$ & 0, 1, 13, 235, 4963, 114630, 2805577, 71506493, 1877836905, 50460010640 \\
$N^{(8)}_{0,2}(n)$ & 0, 1, 12, 212, 4431, 101746, 2480965, 63069818, 1653146798, 44358051118 \\
$N^{(8)}_{0,3}(n)$ & 0, 1, 11, 192, 3991, 91367, 2223657, 56454913, 1478352854, 39639151748 \\
$N^{(8)}_{0,4}(n)$ & 0, 1, 10, 175, 3628, 82951, 2017134, 51181615, 1339684841, 35909134850 \\
$N^{(8)}_{0,5}(n)$ & 0, 1, 9, 161, 3327, 76061, 1849020, 46906469, 1227589222, 32900443583 \\
$N^{(8)}_{0,6}(n)$ & 0, 1, 8, 150, 3073, 70365, 1710137, 43384295, 1135385664, 30428839031 \\
$N^{(8)}_{0,7}(n)$ & 0, 1, 43, 934, 21889, 537062, 13687802, 359087424, 9636561154, 263348238200 \\
$N^{(8)}_{1,1}(n)$ & 0, 1, 12, 210, 4362, 99755, 2425580, 61535661, 1610431956, 43159417675 \\
$N^{(8)}_{1,2}(n)$ & 0, 1, 11, 188, 3865, 87889, 2129428, 53890007, 1407820037, 37678150549 \\
$N^{(8)}_{1,3}(n)$ & 0, 1, 10, 169, 3457, 78400, 1896341, 47936353, 1251248244, 33466671322 \\
$N^{(8)}_{1,4}(n)$ & 0, 1, 9, 153, 3123, 70761, 1710522, 43220709, 1127803986, 30157664021 \\
$N^{(8)}_{1,5}(n)$ & 0, 1, 8, 140, 2848, 64550, 1560212, 39420390, 1028582893, 27503248281 \\
$N^{(8)}_{1,6}(n)$ & 0, 1, 7, 130, 2617, 59450, 1436745, 36306528, 947395098, 25333679568 \\
$N^{(8)}_{1,7}(n)$ & 0, 1, 42, 879, 20229, 490573, 12404611, 323569035, 8645950497, 235480030800 \\
$N^{(8)}_{2,2}(n)$ & 0, 1, 10, 167, 3400, 76907, 1857186, 46893871, 1223035854, 32691728408 \\
$N^{(8)}_{2,3}(n)$ & 0, 1, 9, 149, 3021, 68183, 1644241, 41477837, 1081038867, 28881104104 \\
$N^{(8)}_{2,4}(n)$ & 0, 1, 8, 134, 2713, 61205, 1475473, 37211487, 969667454, 25901997659 \\
$N^{(8)}_{2,5}(n)$ & 0, 1, 7, 122, 2461, 55566, 1339695, 33790620, 880578994, 23523212086 \\
$N^{(8)}_{2,6}(n)$ & 0, 1, 6, 113, 2250, 50964, 1228708, 31000709, 808003087, 21587160025 \\
$N^{(8)}_{2,7}(n)$ & 0, 1, 41, 827, 18763, 451140, 11345248, 294804571, 7855134125, 213478075083 \\
$N^{(8)}_{3,3}(n)$ & 0, 1, 8, 132, 2668, 60114, 1448096, 36502706, 950854362, 25392541071 \\
$N^{(8)}_{3,4}(n)$ & 0, 1, 7, 118, 2383, 53696, 1293418, 32601738, 849192089, 22676570777 \\
$N^{(8)}_{3,5}(n)$ & 0, 1, 6, 107, 2151, 48537, 1169547, 29487155, 768196868, 20516236734 \\
$N^{(8)}_{3,6}(n)$ & 0, 1, 5, 99, 1957, 44350, 1068706, 26957025, 702458302, 18764227912 \\
$N^{(8)}_{3,7}(n)$ & 0, 1, 40, 778, 17476, 417681, 10465338, 271254281, 7214368486, 195790330452 \\
$N^{(8)}_{4,4}(n)$ & 0, 1, 6, 105, 2118, 47752, 1150456, 29001285, 755455356, 20174220067 \\
$N^{(8)}_{4,5}(n)$ & 0, 1, 5, 95, 1903, 42996, 1036409, 26136937, 681026743, 18190236725 \\
$N^{(8)}_{4,6}(n)$ & 0, 1, 4, 88, 1723, 39156, 943877, 23817836, 620806352, 16586072441 \\
$N^{(8)}_{4,7}(n)$ & 0, 1, 39, 732, 16353, 389219, 9728639, 251739537, 6687280923, 181319851956 \\
$N^{(8)}_{5,5}(n)$ & 0, 1, 4, 86, 1702, 38581, 930540, 23479751, 612009962, 16351122208 \\
$N^{(8)}_{5,6}(n)$ & 0, 1, 3, 80, 1533, 35035, 844872, 21334517, 556316007, 14867886262 \\
$N^{(8)}_{5,7}(n)$ & 0, 1, 38, 689, 15379, 364882, 9106097, 235366601, 6247309087, 169286708569 \\
$N^{(8)}_{6,6}(n)$ & 0, 1, 2, 75, 1372, 31745, 764910, 19334695, 504391820, 13485214011 \\
$N^{(8)}_{6,7}(n)$ & 0, 1, 37, 649, 14539, 343903, 8574901, 221460787, 5874990574, 159130751326 \\
$N^{(8)}_{7,7}(n)$ & 0, 1, 72, 2448, 65436, 1739892, 46673556, 1269495684, 34995417066, 976419741366 \\
\bottomrule
\end{longtable}

\section{Complete G\"odel--9 pair catalogue}
\label{app:m9}

\begin{longtable}{@{}l p{0.80\textwidth}@{}}
\toprule
Class & Initial terms, \(n=1,\ldots,10\)\\
\midrule
\endfirsthead
\toprule
Class & Initial terms, \(n=1,\ldots,10\)\\
\midrule
\endhead
$N^{(9)}_{0,0}(n)$ & 0, 1, 16, 338, 8194, 215547, 5984040, 172593933, 5121586720, 155357649659 \\
$N^{(9)}_{0,1}(n)$ & 0, 1, 15, 308, 7353, 191601, 5285384, 151745687, 4487441001, 135756916197 \\
$N^{(9)}_{0,2}(n)$ & 0, 1, 14, 281, 6637, 171888, 4722933, 135225769, 3990848160, 120547182266 \\
$N^{(9)}_{0,3}(n)$ & 0, 1, 13, 257, 6031, 155641, 4266990, 121984421, 3596049662, 108529800837 \\
$N^{(9)}_{0,4}(n)$ & 0, 1, 12, 236, 5520, 142198, 3893996, 111233702, 3277236949, 98864562377 \\
$N^{(9)}_{0,5}(n)$ & 0, 1, 11, 218, 5089, 131002, 3585585, 102387320, 3015799818, 90959168576 \\
$N^{(9)}_{0,6}(n)$ & 0, 1, 10, 203, 4723, 121601, 3327639, 95010859, 2798269098, 84392047714 \\
$N^{(9)}_{0,7}(n)$ & 0, 1, 9, 191, 4407, 113648, 3109343, 88782401, 2614818491, 78859344583 \\
$N^{(9)}_{0,8}(n)$ & 0, 1, 53, 1307, 34621, 958256, 27525219, 813409713, 24581124066, 756287090591 \\
$N^{(9)}_{1,1}(n)$ & 0, 1, 14, 279, 6556, 169195, 4637374, 132531855, 3905832466, 117848064968 \\
$N^{(9)}_{1,2}(n)$ & 0, 1, 13, 253, 5881, 150855, 4118749, 117396341, 3453033519, 104031249560 \\
$N^{(9)}_{1,3}(n)$ & 0, 1, 12, 230, 5313, 135829, 3700789, 105334007, 3095063474, 93174129931 \\
$N^{(9)}_{1,4}(n)$ & 0, 1, 11, 210, 4837, 123470, 3360807, 95593924, 2807514578, 84486787238 \\
$N^{(9)}_{1,5}(n)$ & 0, 1, 10, 193, 4438, 113236, 3081204, 87620149, 2572873787, 77415055193 \\
$N^{(9)}_{1,6}(n)$ & 0, 1, 9, 179, 4101, 104690, 2848524, 81002897, 2378528446, 71566353230 \\
$N^{(9)}_{1,7}(n)$ & 0, 1, 8, 168, 3811, 97500, 2652509, 75440108, 2215320738, 66658989533 \\
$N^{(9)}_{1,8}(n)$ & 0, 1, 52, 1240, 32293, 883994, 25200953, 740636250, 22288582045, 683507599412 \\
$N^{(9)}_{2,2}(n)$ & 0, 1, 12, 228, 5244, 133730, 3637736, 103423506, 3036429522, 91351553119 \\
$N^{(9)}_{2,3}(n)$ & 0, 1, 11, 206, 4711, 119776, 3252106, 92343290, 2708655377, 81434292841 \\
$N^{(9)}_{2,4}(n)$ & 0, 1, 10, 187, 4267, 108361, 2939999, 83438635, 2446552680, 73533540238 \\
$N^{(9)}_{2,5}(n)$ & 0, 1, 9, 171, 3897, 98958, 2684538, 76181121, 2233577754, 67128174152 \\
$N^{(9)}_{2,6}(n)$ & 0, 1, 8, 158, 3586, 91145, 2472884, 70183083, 2057867497, 61850509638 \\
$N^{(9)}_{2,7}(n)$ & 0, 1, 7, 148, 3319, 84605, 2295291, 65160123, 1910843604, 57437723199 \\
$N^{(9)}_{2,8}(n)$ & 0, 1, 51, 1176, 30201, 819721, 23240305, 680370309, 20416228257, 624706113183 \\
$N^{(9)}_{3,3}(n)$ & 0, 1, 10, 185, 4210, 106760, 2894040, 82086229, 2405885928, 72288354947 \\
$N^{(9)}_{3,4}(n)$ & 0, 1, 9, 167, 3795, 96164, 2605513, 73876721, 2164712451, 65029022965 \\
$N^{(9)}_{3,5}(n)$ & 0, 1, 8, 152, 3451, 87476, 2370325, 67211268, 1969449104, 59163945929 \\
$N^{(9)}_{3,6}(n)$ & 0, 1, 7, 140, 3163, 80289, 2176209, 61722039, 1808890060, 54346931881 \\
$N^{(9)}_{3,7}(n)$ & 0, 1, 6, 131, 2916, 74301, 2013886, 57140327, 1674960667, 50331318589 \\
$N^{(9)}_{3,8}(n)$ & 0, 1, 50, 1115, 28330, 764115, 21579544, 630055849, 18869428274, 576515766594 \\
$N^{(9)}_{4,4}(n)$ & 0, 1, 8, 150, 3406, 86277, 2337008, 66250343, 1940949582, 58300030992 \\
$N^{(9)}_{4,5}(n)$ & 0, 1, 7, 136, 3085, 78203, 2118916, 60078549, 1760339219, 52879395206 \\
$N^{(9)}_{4,6}(n)$ & 0, 1, 6, 125, 2817, 71550, 1939497, 55011281, 1612253642, 48439592571 \\
$N^{(9)}_{4,7}(n)$ & 0, 1, 5, 117, 2587, 66031, 1789894, 50793777, 1489057528, 44747899780 \\
$N^{(9)}_{4,8}(n)$ & 0, 1, 49, 1057, 26665, 715959, 20164962, 587665337, 17576327245, 536462926176 \\
$N^{(9)}_{5,5}(n)$ & 0, 1, 6, 123, 2784, 70657, 1915330, 54322407, 1592000496, 47829521803 \\
$N^{(9)}_{5,6}(n)$ & 0, 1, 5, 113, 2533, 64461, 1748309, 49608721, 1454315689, 43703125955 \\
$N^{(9)}_{5,7}(n)$ & 0, 1, 4, 106, 2317, 59343, 1609373, 45695153, 1340034440, 40279596524 \\
$N^{(9)}_{5,8}(n)$ & 0, 1, 48, 1002, 25191, 674141, 18951929, 551604615, 16482515793, 502724715573 \\
$N^{(9)}_{6,6}(n)$ & 0, 1, 4, 104, 2296, 58660, 1591824, 45194564, 1325411756, 39840702322 \\
$N^{(9)}_{6,7}(n)$ & 0, 1, 3, 98, 2091, 53890, 1461894, 41537600, 1218630471, 36642316576 \\
$N^{(9)}_{6,8}(n)$ & 0, 1, 47, 950, 23893, 637654, 17903948, 520628162, 15546774826, 473949490432 \\
$N^{(9)}_{7,7}(n)$ & 0, 1, 2, 93, 1894, 49430, 1339596, 38099965, 1118234066, 33635414604 \\
$N^{(9)}_{7,8}(n)$ & 0, 1, 46, 901, 22756, 605596, 16991710, 493764751, 14737764127, 449126095831 \\
$N^{(9)}_{8,8}(n)$ & 0, 1, 90, 3645, 110850, 3335970, 101031648, 3098791989, 96261603660, 3025356923106 \\
\bottomrule
\end{longtable}

\section{Complete G\"odel--10 pair catalogue}
\label{app:m10}

\begin{longtable}{@{}l p{0.80\textwidth}@{}}
\toprule
Class & Initial terms, \(n=1,\ldots,10\)\\
\midrule
\endfirsthead
\toprule
Class & Initial terms, \(n=1,\ldots,10\)\\
\midrule
\endhead
$N^{(10)}_{0,0}(n)$ & 0, 1, 18, 425, 11486, 336422, 10392468, 333395589, 11001170630, 371015873162 \\
$N^{(10)}_{0,1}(n)$ & 0, 1, 17, 391, 10415, 302298, 9280881, 296411165, 9747929085, 327887519100 \\
$N^{(10)}_{0,2}(n)$ & 0, 1, 16, 360, 9487, 273650, 8367357, 266477562, 8745169486, 293685004678 \\
$N^{(10)}_{0,3}(n)$ & 0, 1, 15, 332, 8687, 249591, 7612789, 242033173, 7933081290, 266160702784 \\
$N^{(10)}_{0,4}(n)$ & 0, 1, 14, 307, 8000, 229339, 6985258, 221868139, 7267033293, 243684214886 \\
$N^{(10)}_{0,5}(n)$ & 0, 1, 13, 285, 7411, 212217, 6459088, 205052841, 6713792546, 225068991467 \\
$N^{(10)}_{0,6}(n)$ & 0, 1, 12, 266, 6905, 197653, 6013901, 190876787, 6248583100, 209445949659 \\
$N^{(10)}_{0,7}(n)$ & 0, 1, 11, 250, 6467, 185180, 5633672, 178797894, 5852849445, 196172758831 \\
$N^{(10)}_{0,8}(n)$ & 0, 1, 10, 237, 6082, 174436, 5305784, 168402165, 5512589508, 184769491911 \\
$N^{(10)}_{0,9}(n)$ & 0, 1, 64, 1767, 52210, 1609354, 51442382, 1690945605, 56824202127, 1943796049422 \\
$N^{(10)}_{1,1}(n)$ & 0, 1, 16, 358, 9394, 270147, 8242120, 262057273, 8589161840, 288154283159 \\
$N^{(10)}_{1,2}(n)$ & 0, 1, 15, 328, 8513, 243281, 7392584, 234389147, 7666542521, 256797886297 \\
$N^{(10)}_{1,3}(n)$ & 0, 1, 14, 301, 7757, 220828, 6694293, 211903849, 6922906360, 231681882646 \\
$N^{(10)}_{1,4}(n)$ & 0, 1, 13, 277, 7111, 202021, 6116350, 193441321, 6315763582, 211262687937 \\
$N^{(10)}_{1,5}(n)$ & 0, 1, 12, 256, 6560, 186198, 5633996, 178113822, 5813604949, 194421463457 \\
$N^{(10)}_{1,6}(n)$ & 0, 1, 11, 238, 6089, 172802, 5227665, 165245760, 5393036738, 180341991056 \\
$N^{(10)}_{1,7}(n)$ & 0, 1, 10, 223, 5683, 161381, 4882039, 154323919, 5036610018, 168423345094 \\
$N^{(10)}_{1,8}(n)$ & 0, 1, 9, 211, 5327, 151588, 4585103, 144958081, 4731209091, 158218194243 \\
$N^{(10)}_{1,9}(n)$ & 0, 1, 63, 1687, 49061, 1496566, 47493719, 1552931653, 51976666236, 1772353045161 \\
$N^{(10)}_{2,2}(n)$ & 0, 1, 14, 299, 7676, 218015, 6598894, 208658395, 6811419666, 227810393628 \\
$N^{(10)}_{2,3}(n)$ & 0, 1, 13, 273, 6961, 196995, 5949389, 187837041, 6125046679, 204685606420 \\
$N^{(10)}_{2,4}(n)$ & 0, 1, 12, 250, 6353, 179469, 5414149, 170811147, 5566870874, 185957066911 \\
$N^{(10)}_{2,5}(n)$ & 0, 1, 11, 230, 5837, 164790, 4969287, 156731484, 5106930458, 170565628678 \\
$N^{(10)}_{2,6}(n)$ & 0, 1, 10, 213, 5398, 152416, 4596004, 144954309, 4723062907, 157741321573 \\
$N^{(10)}_{2,7}(n)$ & 0, 1, 9, 199, 5021, 141910, 4279644, 134992537, 4398797246, 146919173770 \\
$N^{(10)}_{2,8}(n)$ & 0, 1, 8, 188, 4691, 132940, 4008749, 126477308, 4121795098, 137679958973 \\
$N^{(10)}_{2,9}(n)$ & 0, 1, 62, 1610, 46193, 1397374, 44104283, 1436533520, 47942577025, 1631161806542 \\
$N^{(10)}_{3,3}(n)$ & 0, 1, 12, 248, 6284, 177250, 5342216, 168435386, 5486945842, 183224282119 \\
$N^{(10)}_{3,4}(n)$ & 0, 1, 11, 226, 5711, 160856, 4843786, 152627610, 4969804897, 165900176401 \\
$N^{(10)}_{3,5}(n)$ & 0, 1, 10, 207, 5227, 147181, 4431039, 139599855, 4545052640, 151706943858 \\
$N^{(10)}_{3,6}(n)$ & 0, 1, 9, 191, 4817, 135698, 4085898, 128737401, 4191624554, 139915228612 \\
$N^{(10)}_{3,7}(n)$ & 0, 1, 8, 178, 4466, 125985, 3794324, 119576783, 3893915257, 129991414378 \\
$N^{(10)}_{3,8}(n)$ & 0, 1, 7, 168, 4159, 117725, 3545371, 111768303, 3640269324, 121540626179 \\
$N^{(10)}_{3,9}(n)$ & 0, 1, 61, 1536, 43591, 1310201, 41186455, 1337755929, 44554941267, 1513543250093 \\
$N^{(10)}_{4,4}(n)$ & 0, 1, 10, 205, 5170, 145460, 4377160, 137858869, 4487356488, 149755729487 \\
$N^{(10)}_{4,5}(n)$ & 0, 1, 9, 187, 4715, 132664, 3991993, 125723521, 4092206011, 136564267185 \\
$N^{(10)}_{4,6}(n)$ & 0, 1, 8, 172, 4331, 121956, 3670885, 115633268, 3764270024, 125632257609 \\
$N^{(10)}_{4,7}(n)$ & 0, 1, 7, 160, 4003, 112929, 3400369, 107145979, 3488711060, 116453566141 \\
$N^{(10)}_{4,8}(n)$ & 0, 1, 6, 151, 3716, 105281, 3169966, 99929147, 3254475387, 108654534929 \\
$N^{(10)}_{4,9}(n)$ & 0, 1, 60, 1465, 41240, 1233575, 38664634, 1253339969, 41683071924, 1414432974974 \\
$N^{(10)}_{5,5}(n)$ & 0, 1, 8, 170, 4286, 120637, 3630608, 114350843, 3722199742, 124219951812 \\
$N^{(10)}_{5,6}(n)$ & 0, 1, 7, 156, 3925, 110603, 3330116, 104917569, 3415829219, 114012217666 \\
$N^{(10)}_{5,7}(n)$ & 0, 1, 6, 145, 3617, 102170, 3077577, 97000841, 3158940322, 105459147991 \\
$N^{(10)}_{5,8}(n)$ & 0, 1, 5, 137, 3347, 95051, 2862934, 90283597, 2941011288, 98205684080 \\
$N^{(10)}_{5,9}(n)$ & 0, 1, 59, 1397, 39125, 1166129, 36474292, 1180648117, 39225031965, 1329985932146 \\
$N^{(10)}_{6,6}(n)$ & 0, 1, 6, 143, 3584, 101157, 3047410, 96047947, 3127887216, 104421615263 \\
$N^{(10)}_{6,7}(n)$ & 0, 1, 5, 133, 3293, 93241, 2810309, 88618981, 2886900209, 96399997615 \\
$N^{(10)}_{6,8}(n)$ & 0, 1, 4, 126, 3037, 86583, 2609133, 82327533, 2682818640, 89608714264 \\
$N^{(10)}_{6,9}(n)$ & 0, 1, 58, 1332, 37231, 1106601, 34561029, 1117559655, 37101396183, 1257272065903 \\
$N^{(10)}_{7,7}(n)$ & 0, 1, 4, 124, 3016, 85780, 2586544, 81610804, 2659594156, 88834796242 \\
$N^{(10)}_{7,8}(n)$ & 0, 1, 3, 118, 2771, 79530, 2396934, 75685800, 2467393431, 82439462896 \\
$N^{(10)}_{7,9}(n)$ & 0, 1, 57, 1270, 35543, 1053834, 32879628, 1062376482, 35250195876, 1194045382312 \\
$N^{(10)}_{8,8}(n)$ & 0, 1, 2, 113, 2534, 73650, 2217276, 70079445, 2285483186, 76386871064 \\
$N^{(10)}_{8,9}(n)$ & 0, 1, 56, 1211, 34046, 1006776, 31393110, 1013739321, 33622911257, 1138570926131 \\
$N^{(10)}_{9,9}(n)$ & 0, 1, 110, 5225, 178310, 5993790, 202309008, 6908493669, 238793940440, 8347743178346 \\
\bottomrule
\end{longtable}

\vspace{2cm }

\noindent\rule{\textwidth}{0.4pt}

\vspace{12cm}

\noindent\rule{\textwidth}{0.4pt}

\begin{verse}
\itshape
That is, you must take living so seriously\\
that even at seventy, for instance, you'll plant olive trees,\\
not because they'll be left to your children,\\
but because, although you fear death, you do not believe in it,\\
because life weighs more heavily in the balance.\\
$\;$\\
N. Hikmet
\end{verse}


\begin{thebibliography}{9}

\bibitem{YildizGodel}
V.~Yildiz,
\emph{G\"odel Implication on Finite Chains:
Catalan Enumeration and a Continuum Limit Law},
arXiv:2602.16135, 2026.

\bibitem{OEIS}
N.~J.~A.~Sloane and The OEIS Foundation Inc.,
\emph{The On-Line Encyclopedia of Integer Sequences}.

\end{thebibliography}
\end{document}